\documentclass[twoside,leqno]{article}

\usepackage[letterpaper]{geometry}

\usepackage{siamproceedings}

\usepackage[T1]{fontenc}
\usepackage{amsfonts}
\usepackage{graphicx}
\usepackage{epstopdf}
\usepackage{enumitem}
\usepackage{algorithmic}
\usepackage[all]{xy}
\ifpdf
  \DeclareGraphicsExtensions{.eps,.pdf,.png,.jpg}
\else
  \DeclareGraphicsExtensions{.eps}
\fi
\newcommand{\Psib}{\mbox{\boldmath$\Psi$}}

\newsiamremark{remark}{Remark}
\newsiamremark{hypothesis}{Hypothesis}
\crefname{hypothesis}{Hypothesis}{Hypotheses}
\newsiamthm{claim}{Claim}
\newsiamremark{example}{Example}
\newsiamremark{prop}{Proposition}
\newsiamremark{conjecture}{Conjecture}

\usepackage{amsopn}

\begin{document}

\title{\Large 
Representations and characters of quantum affine algebras at the crossroads between cluster categorification and quantum integrable models.}
    \author{David Hernandez\thanks{Université Paris Cit\'e, Sorbonne Universit\'e, CNRS, IMJ-PRG, F-75013 Paris, France
		(\email{david.hernandez@imj-prg.fr}).}
    }

\date{}

\maketitle







\begin{abstract} In this lecture, we survey a number of recent results and developments regarding the representation theory  of infinite-dimensional quantum groups (quantum affine algebras and related algebras), as well as their connections with cluster categorification and quantum integrable models. We will also give new examples and conjectures. 
\end{abstract}

{\it 

}

\section{Introduction}

Quantum affine algebras are fundamental examples of Drinfeld-Jimbo quantum groups (we refer to \cite{Jim0} and to the seminal ICM lecture by Drinfeld \cite{Dr} introducing the subject). Consider a connected simply connected simple algebraic group $G$ over the complex numbers.
The loop algebra 
\begin{equation}\label{lalg}\mathcal{L}\mathfrak{g} = \mathfrak{g}\otimes \mathbb{C}[t^{\pm 1}]\end{equation} 
defined from its Lie algebra $\text{Lie}(G) = \mathfrak{g}$ admits a universal central extension 
$\hat{\mathfrak{g}} = \mathcal{L}\mathfrak{g} \oplus \mathbb{C} c$ which is called an affine Kac-Moody algebra\footnote{We do not consider the derivation element in this lecture.}, see \cite{Kac}. 
The quantum affine algebra $\mathcal{U}_q(\hat{\mathfrak{g}})$ is a quantization of the universal enveloping algebra of $\hat{\mathfrak{g}}$. It is a Hopf algebra with a very rich representation theory which has been studied intensively during the past thirty-five years from different points of view and with various applications.

We will review some of the results on the intricate structure of relevant categories of representations, in particular on the 
structure of Grothendieck rings and on character formulas in terms of analogs of the Kazhdan-Lusztig conjecture.

We will also focus on recent developments regarding two of the main applications of this theory: one which is at the very origin of quantum groups, namely quantum integrable models, and a more recent one connected to cluster algebras and their categorifications. 
Some of the main results in these directions include monoidal categorification theorems and  Baxter polynomiality.

Certain Grothendieck rings of representations of quantum affine algebras have a cluster structure, whose exchange relations 
contain remarkable relations such as $QQ$-relations. The latter appear in the study of quantum integrable models 
and the dual ordinary differential equations, to establish Bethe Ansatz relations among the roots of Baxter polynomials. 
These polynomials are used to express the eigenvalues of transfer-matrices, which are constructed in terms of quantum affine algebras. 
In an informal way we could sum up the content of this circle of ideas discussed in this lecture in the following square:
$$
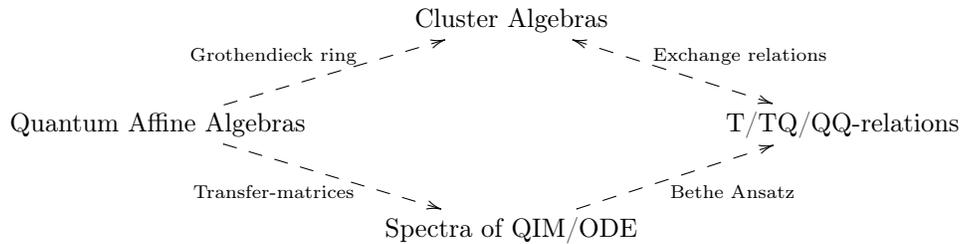
\xymatrix{ &\text{Cluster Algebras}\ar@{<-->}[dr]^{\qquad \text{Exchange relations}}&   
\\ \text{Quantum Affine Algebras} \ar@{-->}[ur]^{\text{Grothendieck ring}\qquad}\ar@{-->}[dr]_{\text{Transfer-matrices}\qquad}& &  \text{ T/TQ/QQ-relations}
\\ & \text{Spectra of QIM/ODE}\ar@{-->}[ur]_{\qquad\text{Bethe Ansatz}}& }$$
It will be stated precisely and illustrated in the lecture with examples and results. A part of the content is also motivated by 
the geometric Langlands program, quiver varieties and Coulomb branches as explained below. Additional recent subjects will include shifted quantum affine algebras and their truncations.

\section{Quantum affine algebras and their representations}

Consider a complex simple Lie algebra $\mathfrak{g}$ of rank $n$ (for example, the Lie algebra $\mathfrak{g} = \mathfrak{sl}_2$ 
of rank $n = 1$ is already very interesting for the topics in this lecture). 
The untwisted affine Kac-Moody algebra $\hat{\mathfrak{g}}$ is the universal 
central extension of the loop algebra (\ref{lalg}).

For $q\in\mathbb{C}^*$ a quantum parameter (which is assumed to be generic, that is not a root of unity), we have the corresponding quantum affine algebra $\mathcal{U}_q(\hat{\mathfrak{g}})$ (also called affine quantum group). It is a Hopf algebra and a $q$-deformation of the universal enveloping algebra $\mathcal{U}(\hat{\mathfrak{g}})$ of $\hat{\mathfrak{g}}$. This is one of the most important quantum groups in the sense of Drinfeld-Jimbo.

The quantum affine algebra $\mathcal{U}_q(\hat{\mathfrak{g}})$ has a relatively elementary presentation: 
by generators $e_i,\ f_i,\ k_i^{\pm 1}$ ($0\le i\le n$) with the
following relations for $0\leq i,j\leq n$:
\begin{align*}
&k_ik_j=k_jk_i,\quad
&k_ie_jk_i^{-1}=q_i^{C_{i,j}}e_j,\quad k_if_jk_i^{-1}=q_i^{-C_{i,j}}f_j,
\\
&[e_i,f_j]
=\delta_{i,j}\frac{k_i-k_i^{-1}}{q_i-q_i^{-1}},
\\
&\sum_{r=0}^{1-C_{i.j}}(-1)^re_i^{(1-C_{i,j}-r)}e_j e_i^{(r)}=0\quad (i\neq j),
&\sum_{r=0}^{1-C_{i.j}}(-1)^rf_i^{(1-C_{i,j}-r)}f_j f_i^{(r)}=0\quad (i\neq j),
\end{align*}
where $C = (C_{i,j})_{0\leq i,j\leq n}$ is the Cartan matrix of $\hat{\mathfrak{g}}$, 
$D=\mathrm{diag}(d_0,\ldots,d_n)$ is the unique diagonal matrix such
that $B=DC$ is symmetric, the $d_i$'s are relatively prime positive integers and $q_i = q^{d_i}$. 
For $0\le i \le n$ we have set $x_i^{(r)}=x_i^r/[r]_{q_i}!$ ($x_i=e_i,f_i$), with the standard notation for quantum factorials. 
The last two relations are called quantum Serre relations.

\begin{example} For $\mathfrak{g} = \mathfrak{sl}_2$, $C = \begin{pmatrix}2&-2\\-2&2\end{pmatrix}$, we have generators $e_0,e_1,f_0,f_1,k_0^{\pm 1},k_1^{\pm 1}$ with $q_0 = q_1 = q$. One of the quantum Serre relations is
$$e_0^3e_1 - (q^2 + 1 + q^{-2})e_0^2e_1e_0 + (q^2 + 1 + q^{-2}) e_0e_1e_0^2 - e_1e_0^3 = 0.$$
\end{example}


The subalgebra of $\mathcal{U}_q(\hat{\mathfrak{g}})$ generated by the $e_i$, $f_i$, $k_i^{\pm 1}$ with $i\in I = \{1,\cdots n\}$ is isomorphic to 
the quantum group of finite type $\mathcal{U}_q(\mathfrak{g})$. The algebra $\mathcal{U}_q(\hat{\mathfrak{g}})$ has a Hopf algebra structure with the coproduct defined by
\begin{align*}
&\Delta(e_i)=e_i\otimes 1+k_i\otimes e_i,\quad
\Delta(f_i)=f_i\otimes k_i^{-1}+1\otimes f_i,
\quad
\Delta(k_i)=k_i\otimes k_i\,\text{ for $0\leq i\leq n$.}
\end{align*}

One reason for the importance of $\mathcal{U}_q(\hat{\mathfrak{g}})$ is the depth of its representation theory and the plethora of its applications. In this lecture, in addition to the pure understanding of the representations of quantum affine algebras and related algebras, we will focus mainly on two applications: the corresponding quantum integrable models and the categorification of cluster algebras, as well as the relations between these two subjects.

Let $\mathcal{C}$ be the category of finite-dimensional representations of $\mathcal{U}_q(\hat{\mathfrak{g}})$. 
This monoidal category is not semi-simple and not braided. All simple finite-dimensional representations have level $0$, that is, the central element $c = k_0^{a_0}k_1^{a_1}\cdots k_n^{a_n}$ 
acts as the identity or its opposite (here the $a_i$ are the Kac labels; for $\mathfrak{g} = \mathfrak{sl}_2$, we have $c = k_0k_1$). The simple objects of $\mathcal{C}$ have been parametrized in terms of Drinfeld polynomials by Chari-Pressley (see \cite{CP2} and \cite{CH} for a review): 

\begin{theorem} There is a bijection between simple representations\footnote{In this lecture, we will only consider type I representations of quantum enveloping algebras, that is, the eigenvalues of the $k_i$ belong to $q^{\mathbb{Z}}$. Any simple finite-dimensional representation can be obtained from a type I representation by twisting with some signs, see \cite[10.1]{CP2}.}  in $\mathcal{C}$ and $n$-tuples $(P_i(z))_{i\in I}$ of polynomials with complex coefficients and constant term $1$. 
\end{theorem}

\begin{example} For $i\in I$ and $a\in\mathbb{C}^*$, we have the fundamental representation $V_i(a)$ which corresponds to 
the $n$-tuple of Drinfeld polynomials $(1,\cdots, 1, \underbrace{1 - za}_{\text{position }i}, 1, \cdots, 1)$.
\end{example}

\begin{example}\label{evalj} Let $\mathfrak{g} = \mathfrak{sl}_2$ and consider the finite-type quantum group $\mathcal{U}_q(\mathfrak{sl}_2)$ with generators $E$, $F$, $K^{\pm 1}$ and the usual relations $KE = q^2EK$, $KF = q^{-2}FK$, $[E,F] = \frac{K - K^{-1}}{q - q^{-1}}$.
For every $a\in\mathbb{C}^*$, we have a surjective algebra homomorphism
$ev_a : \mathcal{U}_q(\widehat{\mathfrak{sl}_2}) \rightarrow \mathcal{U}_q(\mathfrak{sl}_2)$ such that:
$$ev_a(e_1) = E,\quad ev_a(f_1) = F,\quad ev_a(e_0) = q^{-1}aF,\quad ev_a(f_0) = qa^{-1}E.$$
The simple finite-dimensional $\mathcal{U}_q(\mathfrak{sl}_2)$-modules are known to be classified by their dimension: for every $N\in \mathbb{Z}_{\ge 0}$ there is a unique (up to isomorphism) simple module $V_N$ of dimension $N+1$.
Therefore, pulling back by the evaluation morphisms $ev_a$, we get a one-parameter family of simple $\mathcal{U}_q(\widehat{\mathfrak{sl}_2})$-modules $V_N(a)$ ($a\in \mathbb{C}^*$) of dimension $N+1$. The representations $V_0(a)$ are all equal to the trivial representation. Otherwise, for $N\geq 1$, the simple modules $V_N(a)$ and $V_N(b)$ are non-isomorphic if $a \not = b$. 
The modules $V_n(a)$ are called \emph{evaluation modules}, and the $V_1(a)$ are the fundamental representations.
\end{example}

\begin{example}\label{evaljg} More generally, there exists an evaluation morphism as in the previous example for $\mathfrak{g}$ of type $A$ by \cite{Jim}, but not for a general finite type. This is one of the main differences with the classical case ($ q = 1$) and a source of many technical difficulties. For example, let $\mathfrak{g}$ be of type $D_4$. The fundamental representation $V_2(a)$ of $\mathcal{U}_q(\hat{\mathfrak{g}})$ associated to the trivalent node $i = 2$ and $a\in \mathbb{C}^*$ has dimension $29$ and is not simple as a $\mathcal{U}_q(\mathfrak{g})$-module (it decomposes as a direct sum of the fundamental representation of dimension $28$ and the trivial representation). This is an obstruction to the existence of an evaluation morphism from $\mathcal{U}_q(\hat{\mathfrak{g}})$ to $\mathcal{U}_q(\mathfrak{g})$ in this case.
\end{example}

\begin{theorem}\label{stgr}\cite{Fre} The Grothendieck ring 
$K(\mathcal{C})$ of $\mathcal{C}$ is commutative and polynomial 
$$K(\mathcal{C})\simeq \mathbb{Z}[X_{i,a}]_{i\in I, a\in\mathbb{C}^*},$$
with $X_{i,a}$ is the class $[V_{i,a}]$ class of the fundamental representation $V_i(a)$.
\end{theorem}

The reader may refer to the recent lecture \cite{H} for generalities on Grothendieck rings. The commutativity of $K(\mathcal{C})$ means that the Jordan-H\"older multiplicities of simple modules are the same in $V\otimes V'$ 
and in $V'\otimes V$. But these modules are not isomorphic in general. For example for $\mathfrak{g} = sl_2$, 
$V_1(q^2)\otimes V_1(1)$ and $V_1(1)\otimes V_1(q^2)$ are not isomorphic. Indeed $V_1(q^2)\otimes V_1(1)$ has a unique proper submodule, which has dimension $1$, and $V_1(1)\otimes V_1(q^2)$ has a unique proper submodule, which has dimension $3$.

Although, as a ring, the structure of $K(\mathcal{C})$ is understood by the Theorem \ref{stgr}, it is much more intricate as a based ring: the basis of simple classes is difficult to describe in terms of the basis of monomials in the fundamental classes (also called the standard basis). In fact, this 
description is not known in general. This is closely related to the intricate combinatorics of cluster algebras as discussed below. For the simply-laced types, there is an answer based on the study of perverse sheaves on quiver varieties by Nakajima (see Theorem \ref{klsl} below).

The quantum affine algebra also has many interesting other infinite-dimensional simple representations, for example the evaluation of simple Verma modules of $\mathcal{U}_q(\widehat{\mathfrak{sl}_2})$ (they are also of level $0$). They belong to the category $\hat{\mathcal{O}}$ that will be discussed below.

\begin{remark}The representations of $\mathcal{U}_q(\hat{\mathfrak{g}})$ are closely related to representations of Yangians, see \cite{GTL}.\end{remark}

\section{Quantum integrable models}\label{qims}

The historical motivation to study the representation theory of affine quantum groups is closely related to quantum integrable models, more precisely to $XXZ$-models and their variations. Recall that, by Drinfeld \cite{Dr}, the quantum affine algebra $\mathcal{U}_q(\hat{\mathfrak{g}})$ has the fundamental property of having a universal $R$-matrix.

\begin{theorem} There exists a non-trivial formal power series
$$\mathcal{R}(z)\in \mathcal{U}_q(\hat{\mathfrak{g}})^{\hat{\otimes} 2}[[z]]$$
with coefficients in a completion of the tensor square of $\mathcal{U}_q(\hat{\mathfrak{g}})$, satisfying several properties including the quantum Yang-Baxter equation
$$\mathcal{R}_{12}(z)\mathcal{R}_{13}(zw)\mathcal{R}_{23}(w) = \mathcal{R}_{23}(w)
\mathcal{R}_{13}(zw) \mathcal{R}_{12}(z)$$
with the standard notations $\mathcal{R}_{12}(z) = \mathcal{R}(z)\otimes 1$, $\mathcal{R}_{23}(z) = 1 \otimes \mathcal{R}(z)$, $\mathcal{R}_{13}(z) = (P\otimes \text{Id})(\mathcal{R}_{23}(z))$ where $P$ is the flip operator $P(x\otimes y) = y\otimes x$.
\end{theorem}

The existence of the universal $R$-matrix has many important consequences, in particular it enables the construction of quantum integrable systems 
(see \cite{Hb} for a review). Let us recall the main points. Let $V$ be a finite-dimensional representation of $\mathcal{U}_q(\hat{\mathfrak{g}})$ that we call 
``auxiliary space". The corresponding transfer-matrix is:
\begin{equation}\label{transfer}
\mathcal{T}_V(z) = ((\operatorname{Tr}_V \circ \rho_V) \otimes \operatorname{id})({\mathcal{R}(z)})\in \mathcal{U}_q(\hat{\mathfrak{g}})[[z]],
\end{equation}
where $\operatorname{Tr}_V$ is the trace defined on $\text{End}(V)$ and $\rho_V$ is the algebra morphism corresponding the representation $V$.

\begin{remark}\label{trace} It is also useful to introduce a twist in the definition of the transfer-matrices: the 
auxiliary space $V$ has a natural grading 
$V = \bigoplus_{\omega = (\omega_1,\ldots,\omega_n)\in\mathbb{Z}^n} V_\omega$ by finite-dimensional weight spaces which are the common eigenspaces of the $k_i^{\pm 1}$ 
($n$ is the rank of $\mathfrak{g}$). Then the trace $\operatorname{Tr}_V$ 
is replaced by $\sum_{\omega\in\mathbb{Z}^n} (u_1^{\omega_1}\ldots u_n^{\omega_n})\text{Tr}_{V_\omega}$
for formal variables $u_i$ with $i\in I$. We obtain the twisted transfer-matrix $\mathcal{T}_{u,V}(z)$. 
This will be useful for handling infinite-dimensional auxiliary spaces below.
\end{remark}

\begin{prop} The transfer-matrices commute: for $V$ and $V'$ in $\mathcal{C}$ we have
$$\mathcal{T}_V(z)\mathcal{T}_{V'}(z') = \mathcal{T}_{V'}(z')\mathcal{T}_V(z)\text{ in }\mathcal{U}_q(\hat{\mathfrak{g}})[[z,z']].$$
The same holds for twisted transfer-matrices for a fixed twist parameter $u =(u_1,\cdots, u_n)$.
\end{prop}

This fundamental ``integrability" property follows from the quantum Yang-Baxter equation. It means that the coefficients $\mathcal{T}_V[N]\in\mathcal{U}_q(\hat{\mathfrak{g}})$ of all the transfer-matrices $\mathcal{T}_V(z) = \sum_{N\geq 0}z^N \mathcal{T}_V[N]$ commute.

Also, the properties of the trace and of $\mathcal{R}(z)$ imply that the map $V\mapsto \mathcal{T}_V(z)$ defines a ring morphism 
\begin{equation}\label{take}\mathcal{T}(z) : K(\mathcal{C})\rightarrow (\mathcal{U}_q(\hat{\mathfrak{g}}))[[z]].\end{equation}
This ring morphism is injective (see \cite{Fre}).

Now consider $W$ another finite-dimensional representation in $\mathcal{C}$ ($W$ is called the ``quantum space"). Each transfer-matrix $\mathcal{T}_V(z)$ defines an operator on $W[[z]] = W\otimes \mathbb{C}[[z]]$. Hence, $W[[z]]$ has the structure of a representation of the Grothendieck ring $K(\mathcal{C}$). The coefficients $\mathcal{T}_V[N]$ act on $W$ by a large family of commutative operators. This is a quantum system and the next step is to study the spectrum of the quantum system, that is, the eigenvalues of the transfer-matrices on $W$.

\begin{example}
The case where $W$ is a tensor product of finite-dimensional simple representations of $\mathcal{U}_q(\hat{\mathfrak{g}})$ is remarkable.  Indeed for ${\mathfrak g} = \mathfrak{sl}_2$, $V$ a $2$-dimensional fundamental representation and the quantum space $W$ a tensor product of fundamental representations, we recover the historical case of the $XXZ$ model (spin chain model).
 The image of $\mathcal{T}_{V}(z)$ in $\text{End}(W)[[z]]$ is the original transfer-matrix defined by Baxter to 
solve the model\footnote{Baxter introduced the powerful method of "Q-operators", which he used to solve the more involved 8 vertex model. The 6 vertex model has also been solved with other methods, in particular in the works of Lieb and Sutherland (1967).}. 
\end{example}

By \cite{Fre}, each representation $V$ in $\mathcal{C}$ has a $q$-character $\chi_q(V)\in \mathbb{Z}[Y_{i,a}^{\pm 1}]_{i\in I, a\in\mathbb{C}^*}$. It can be defined from a limit of the twisted transfer-matrix $\mathcal{T}_{u,V}(z)$, or by analyzing the decomposition into common generalized eigenspaces of a family of elements in $\mathcal{U}_q(\hat{\mathfrak{g}})$ whose action on $V$ commute (the Drinfeld generators $h_{i,r}$ with $i\in I, r < 0$). Although the $q$-character of a general simple module in $\mathcal{C}$ is not known, 
the $q$-character $\chi_q(V_i(a))$ of fundamental representations can be computed from an algorithm \cite{FM}.

\begin{example}\label{exqc} For $\mathfrak{g} = \mathfrak{sl}_2$, we have $\chi_q(V_1(a)) = Y_{1,a} + Y_{1,aq^2}^{-1}$.
\end{example}

The following polynomiality extends results of Baxter (for $\mathfrak{g} = \mathfrak{sl}_2$) and establishes a conjecture of \cite{Fre}.

\begin{theorem}\label{bthm}\cite{FH} The eigenvalues $\lambda_k$ of $\mathcal{T}_{V,u}(z)$ on $W[[z]]$ are obtained from $\chi_q(V)$ by replacing each $Y_{i,a}$ by a quotient $y_{i,k}\frac{f_i(azq_i^{-1})Q_{i,k}(zaq_i^{-1})}{f_i(azq_i)Q_{i,k}(zaq_i)}$ where $Q_{i,k}(z)$ is polynomial, $y_{i,k}$ is constant and
the functions $f_i(z)$ depend only on $W$.
\end{theorem}

To prove this theorem,  we use infinite-dimensional representations as
explained in the next section.

\section{Quantum affine Borel} Consider now the quantum affine Borel $\mathcal{U}_q(\hat{\mathfrak{b}})$, the subalgebra of $\mathcal{U}_q(\hat{\mathfrak{g}})$  generated by the $e_i, k_i^{\pm 1}$ with $0\leq i\leq n$. $\mathcal{U}_q(\hat{\mathfrak{b}})$ has the crucial property of being a Hopf subalgebra of $\mathcal{U}_q(\hat{\mathfrak{g}})$. The proof of Theorem \ref{bthm} is based on the study of a category $\mathcal{O}$ of representations of $\mathcal{U}_q(\hat{\mathfrak{b}})$ which was introduced in \cite{HJ}. The category $\mathcal{O}$ is the category of $\mathcal{U}_q(\hat{\mathfrak{b}})$-modules with a diagonal action of the $k_i$ and 
finite-dimensional weight spaces (plus a usual technical cone condition on the weights). It is a monoidal category. The simple 
objects in $\mathcal{C}$ (the simple finite-dimensional representations of $\mathcal{U}_q(\hat{\mathfrak{g}})$) are still irreducible when the action is restricted to $\mathcal{U}_q(\hat{\mathfrak{b}})$. Hence we can see 
the simples in $\mathcal{C}$ as simples in $\mathcal{O}$ and we can see $K(\mathcal{C})$ as a (based) subring of $K(\mathcal{O})$.

\begin{theorem}\cite{HJ} The simple objects $L^{\mathfrak{b}}(\Psib)$ in $\mathcal{O}$ are parametrized by $n$-tuples of rational fractions $\Psib = (\Psib_i(z))_{i\in I}$ regular at $0$ (without pole or zero at $0$).
\end{theorem}

\begin{example}\label{expre} Let $i\in I$ and $a\in\mathbb{C}^*$. Set $\Psib_{i,a} = (1 - \delta_{i,j}za)_{j\in I} = (1,\cdots, 1, \underbrace{1 - za}_{\text{position }i}, 1, \cdots, 1)$.

(i) $L_{i,a}^- = L^{\mathfrak{b}}(\Psib_{i,a}^{-1})$ is called a negative prefundamental representation and can be constructed as a limit of simple modules in $\mathcal{C}$ 
(by using a linear inductive system \cite{HJ, Hinv}). The action of $\mathcal{U}_q(\hat{\mathfrak{b}})$ can not be extended to an action of the whole quantum affine algebra (but we will see in Section \ref{shrep} that it
can be extended to an action of a shifted quantum affine algebra). 

(ii) $L_{i,a}^+ = L^{\mathfrak{b}}(\Psib_{i,a})$ is called a positive prefundamental representation and can be obtained from $L_{i,a}^-$ by a certain duality procedure. Together with negative prefundamental representations, these representations play the role of fundamental representations in $\mathcal{O}$ (any simple is a subquotient of a tensor product of prefundamental representations, up to an invertible factor). In the $sl_2$-case, the representation $L_{1,1}^+$ was known as a $q$-oscillator representation 
by Bazhanov-Lukyanov-Zamolodchikov \cite{BLZ}.

(iii) $V_{i,a} = L^{\mathfrak{b}}([\omega_i]\Psib_{i,aq_i^{-1}}\Psib_{i,aq_i}^{-1})$ is a fundamental representation in $\mathcal{C}$. 
Here  $[\omega_i] = (q_i^{\delta_{i,j}})_{j\in I}$ is a multiplicative notation for the fundamental weight $\omega_i$ of $\mathfrak{g}$ (it corresponds to a one dimensional invertible representation in $\mathcal{O}$).
\end{example}

The objects in $\mathcal{O}$ are not necessarily of finite length (the subcategory of finite length objects is not stable by tensor product, for instance $L_{i,a}^+\otimes L_{j,b}^-$ is of infinite length for any $a,b\in\mathbb{C}^*$ and any $i,j\in I$).

\begin{theorem}\cite{FH}\label{gbt} Let $V$ in $\mathcal{C}$. Replace every variable $Y_{i,a}$ appearing in the $q$-character $\chi_q(V)$ with 
$$Y_{i,a}\mapsto  [\omega_i] \frac{[L_{i,aq_i^{-1}}^+]}{[L_{i,aq_i}^+]}.$$
By equating the resulting expression with $[V]$ and clearing the denominators, we obtain a relation in $K(\mathcal{O})$.
\end{theorem}

\begin{example} For $\mathfrak{g} = \mathfrak{sl}_2$, we obtain the original Baxter $TQ$-relation
\begin{equation}\label{tqr}[V_1(a)][L_{1,aq}^+] = [\omega] [L_{1,aq^{-1}}^+]  + [-\omega][L_{1,aq^3}^+],\end{equation}
which corresponds to the $q$-character in Example (\ref{exqc}) (the relation is known in this case by \cite{BLZ}). 
Hence the relations in Theorem \ref{gbt} are the generalized Baxter $TQ$-relations. These relations imply in 
particular that $V_1(a)\otimes L_{1,aq}^+$ has length $2$.
\end{example}

\begin{remark} Theorem \ref{gbt} implies Theorem \ref{bthm}. Indeed, as the left tensor factors of the universal $R$-matrix $\mathcal{R}(z)$ belong to 
$\mathcal{U}_q(\hat{\mathfrak{b}})$, the (twisted) transfer-matrices $\mathcal{T}_{V,u}(z)$ are well-defined for objects $V$ in $\mathcal{O}$. In particular, we have the $Q$-operator $\mathcal{Q}_i(z) = \mathcal{T}_{L_{i,1}^+,u}(z)$ associated to the prefundamental representation $L_{i,1}^+$. It is then proved that the action of a $Q$-operator on the quantum space $W[[z]]$ is polynomial up to a scalar factor 
(note that the $Q$-operator used in \cite{FH} involves a representation $R_{i,1}^+$ dual to $L_{i,1}^+$, but the statements hold for both definitions).  Together with the generalized Baxter $TQ$-relations, this gives the result (see \cite{FH}).
\end{remark}

\section{ODE/IM and Bethe Ansatz}

The ODE/IM correspondence (Ordinary Differential Equations/Integrable models) was discovered at the end of the 90's (Dorey-Tateo, Bazhanov-Lukyanov-Zamolodchikov) and gives a surprising relation between functions associated to Schr\"odinger differential operators and the 
spectrum of quantum systems called "quantum KdV". 
Feigin-Frenkel \cite{FF} have proposed a large generalization and an interpretation of this correspondence in terms of Langlands duality. The Schr\"odinger operators are generalized to affine opers (without monodromy), associated to the Langlands dual affine Lie algebra of the affine Lie algebra which is attached to the quantum KdV system.
This conjecture is largely open, but it is a fruitful source. 
In particular, a remarkable system of relations (the $QQ$-system) was observed \cite{mrv} to be satisfied
by spectral determinants of certain solutions of affine opers (their additive analogues first appeared in \cite{mv}). 
The $QQ$-system of \cite{mrv} for two families of functions $(Q_i(z))_{i\in I}$, $(\tilde{Q}_i(z))_{i\in I}$ is
$$ Q_i(zq_i^{-1})\tilde{Q}_i(zq_i) - Q_i(zq_i)\tilde{Q}_i(zq_i^{-1}) = K \prod_{j|C_{i,j < 0}} Q_j(zq^{C_{i,j}+1})Q_j(zq^{C_{i,j} + 3})\cdots Q_j(zq^{-C_{i,j} - 1}).$$
Depending on the context, $K$ might be a constant or a polynomial in $z$. Motivated by the ODE/IM correspondence, we obtained the following.

\begin{theorem}\label{qqo}\cite{FH2} There is a family of simple classes $[\tilde{L}_{i,a}]$ ($i\in I, a\in\mathbb{C}^*$) in the Grothendieck ring $K(\mathcal{O})$ such that $Q_i(z) = [L_{i,z}^+]$, $\tilde{Q}_i(z) = [\tilde{L}_{i,z}]$ 
satisfy the $QQ$-system in $K(\mathcal{O})$ (with $K$ constant).
\end{theorem}

We have $\tilde{L}_{i,a} =  L^{\mathfrak{b}}(\tilde{\Psib}_{i,aq_i^{-2}})$ (up to a twist by an invertible representation) where 
\begin{equation}\label{tps}\tilde{\Psib}_{i,a} = \Psib_{i,a}^{-1}\left(\prod_{j,C_{i,j} = -1}\Psib_{j,aq_i} \right)\left(\prod_{j,C_{i,j} = -2}\Psib_{j,a}\Psib_{j,aq^2} \right)\left(\prod_{j,C_{i,j} = -3}\Psib_{j,aq^{-1}}\Psib_{j,aq}\Psib_{j,aq^3} \right).\end{equation}
The $QQ$-systems\footnote{The $QQ$-systems for twisted quantum affine algebras have been established in \cite{W}.} imply the Bethe Ansatz equations, which are crucial informations on the root of the Baxter's polynomials, as conjectured 
by various authors (see \cite{Fre, H}): recall $Q_{i,k}(z)$ as in Theorem \ref{bthm} and consider $w$ a (generic) root of $Q_{i,k}$, then 
$$v_i^{-1} \prod_{j \in I}
\frac{{Q}_{j,k}(wq^{B_{i,j}})}{{ Q}_{j,k}(wq^{-B_{i,j}})} = -1,$$
where $v_i = \prod_{j\in I}u_j^{C_{i,j}}$ and $B$ is the symmetrized Cartan matrix.

\begin{example} In the $sl_2$-case, the $QQ$-system is nothing else than the quantum Wronskian relation
$$[L_{1,a}^+][L_{1,a}^-] - [-\omega_1]^2[L_{1,aq^2}^+][L_{1,aq^{-2}}^-] = \chi\text{ where }\chi = (1 - [-\omega_1]^2)^{-1}.$$
The Bethe Ansatz equation in this case reads
$$\frac{{Q}_{1,k}(wq^2)}{{ Q}_{1,k}(wq^{-2})} = - v_1.$$
\end{example}

\begin{example} In the dual context of opers, the $QQ$-system may be expressed with polynomial solutions and an additional polynomial factor $K$ on the 
left hand-side (see for instance \cite[Section 5.2]{FKSZ}). This can also be realized in the present context after renormalization. For example, 
let $\mathfrak{g} = \mathfrak{sl}_2$ and $W = V(q^{-1})$. Let $f_+(z)$ (resp. $f_-(z)$) be the eigenvalue of 
$\mathcal{T}_{u,L_{1,1}^+}(z)$ (resp. $\mathcal{T}_{u,L_{1,1}^-}(z)$) on a highest (resp. lowest) weight vector of $W$ at the limit $v_i = 0$. Then 
$\mathcal{T}_{u,L_{1,1}^\pm}(z) = f_\pm(z) Q_\pm(z)$ with $Q_\pm(z)$ operator which depends on $z$ and $u$ which is a polynomial in $z$ 
(this is the result of \cite{FH} for $Q_+$ and of \cite{Z} for $Q_-$). Then the $QQ$-system implies 
$$\mathcal{T}_{u,L_{1,1}^+}(z)\mathcal{T}_{u,L_{1,1}^-}(z) - u \mathcal{T}_{u,L_{1,q^2}^+}(z)\mathcal{T}_{u,L_{1,q^{-2}}^-}(z) = (1 - u) ^{-1}.$$
Evaluating on the lowest weight vectors, one gets $f_+(z)(1-z) f_-(z) = 1$ and $f_+(zq^{-1})f_+(zq) = (1 - zq^{-1})^{-1}$ 
$ f_-(z) = f_+(zq^{2})$. Hence
$$Q_+(z)Q_-(z) - u  Q_+(zq^2)Q_-(zq^{-2}) = (1 - u)^{-1} (1 - z).$$
\end{example}

\begin{remark}\label{addqq} The genericity condition in the Bethe Ansatz equation was dropped in \cite{FJMM} by using the $QQ^*$-systems which emerge from the 
cluster \cite{HL2} (see the next Section):
$$[L_{i,a}^*][L_{i,a}^+] = \prod_{j,C_{i,j} < 0}[L_{j,aq^{-B_{j,i}}}^+] + [-\alpha_i]\prod_{j,C_{i,j}<0}[L_{j,aq^{B_{i,j}}}^+]\text{ for $i\in I$, $a\in\mathbb{C}^*$,}$$
with $[-\alpha_i] = \prod_{j\in I}[\omega_j]^{-C_{j,i}}$ of dimension $1$ and $L_{i,a}^*$ is the simple module in the category $\mathcal{O}$ associated
 to $\mbox{\boldmath$\Psi$}_{i,a}^{-1}\prod_{j,C_{i,j} < 0}\mbox{\boldmath$\Psi$}_{j,aq^{-B_{j,i}}}$. \end{remark}

\begin{example} In the $sl_2$-case, the $QQ^*$-system and the $TQ$-relation (\ref{tqr}) match, but they differ in general.
\end{example}

This a point where two important fields connected to quantum affine algebras, namely cluster algebras and quantum integrable models, collide. 
The $QQ$-relations themselves will play a role in this connection.

\section{Cluster structures}

The theory of cluster algebras was introduced by Fomin-Zelevinsky \cite{FZ}. A cluster algebra is a commutative algebra 
with a distinguished set of generators grouped into overlapping subsets (the clusters). Each element of the cluster algebra can be expressed as a rational fraction in the elements of a given cluster. More precisely, the cluster algebra $\mathcal{A}_Q$ is a commutative algebra associated with a quiver $Q$ (without loops or 2-cycles). It has a distinguished set of generators (the cluster variables) defined by a combinatorial process.
The cluster algebra $\mathcal{A}_Q$ attached to the quiver
$Q$ (with set of vertices $Q_0$) is a subring of the field 
$$\mathcal{F} = \mathbb{Q}(X_i)_{i\in Q_0}$$
with free variables $X_i$ which are called the initial cluster variables (together with the initial quiver $Q$, the initial cluster variables form the initial cluster). The cluster algebra 
$\mathcal{A}_Q$ is defined as the subalgebra of $\mathcal{F}$ generated by the cluster variables, 
obtained inductively from the initial cluster variables by an inductive process called mutations. 
For example, the first step mutated variables $X_i^*$ are defined by the formula: 
$$X_i X_i^* = \prod_{j\rightarrow i}X_j  + \prod_{j\leftarrow i}X_j,$$
were the arrows $\rightarrow$ and $\leftarrow$ are arrows in the initial quiver $Q$ (the quiver gets also mutated in the process). 
The number of cluster variables is not necessarily finite.
In addition, the cluster variables are grouped into overlapping subsets called clusters, which are all in bijection with $Q_0$. 
The cluster monomials are defined as the monomials into the cluster variables of the same cluster. In some situations, one 
may have additional non-mutable cluster variables: they are called frozen variables and they belong to all clusters (the number 
of clusters does not depend on the frozen cluster variables).

One of the fundamental properties of cluster algebras is the Laurent phenomenon: any cluster variable can be expressed as a Laurent polynomial
in the cluster variables of a given seed. For the initial seed, this can be written as $\mathcal{A}_Q\subset \mathbb{Z}[X_i^{\pm 1}]_{i\in Q_0}$.

\begin{example}\label{ex7} Consider a quiver $Q$ of type $A_2$ $\bullet \longrightarrow \bullet$. We have the initial cluster variables $(X_1,X_2)$, and five cluster variables $X_1 , X_2, \frac{1 + X_2}{X_1}, \frac{1 + X_1}{X_2}  
, \frac{1 + X_1 + X_2}{X_1X_2}$.
There are also five clusters: 
$$(X_1,X_2), \left(\frac{1 + X_2}{X_1}, X_2\right), 
\left(\frac{1 + X_2}{X_1}, \frac{1 + X_1 + X_2}{X_1X_2}\right),\left(\frac{1 + X_1}{X_2},\frac{1 + X_1 + X_2}{X_1X_2}\right),  \left(\frac{1 + X_1}{X_2}, X_1\right).$$
\end{example}

\begin{definition}\cite{HL1} A monoidal category $\mathcal{M}$ is a {\it monoidal categorification} of a cluster algebra $\mathcal{A}$ if there is
$$\phi : \mathcal{A}\rightarrow K(\mathcal{M})$$
a ring isomorphism so that the cluster monomials are sent to simple classes.
\end{definition}

 In the definition above, the simple modules corresponding to 
cluster monomials are called reachable modules. They have the property to be real (the tensor square is also simple), as the square of a cluster monomial is a cluster monomial.

A monoidal categorification can be seen as a cluster symmetry of the Grothendieck ring, see \cite{H}.
It gives useful informations on the monoidal category: for example, it points out remarkable simple representations (corresponding to cluster variables) and it gives a tensor factorization of simple representations corresponding to cluster monomials.
The original examples \cite{HL1} were obtained from the finite-dimensional representations of quantum affine algebras (see references and reviews in \cite{Kel, HLr, H}). 
These are motivated by the $T$-system which are remarkable relations between simple classes in $K(\mathcal{C})$ established in \cite{N, HCr} (and which imply the proof 
of the Kirillov-Reshetikhin conjecture on explicit character formulas for certain simple representations in $\mathcal{C}$). 
These $T$-systems can then be interpreted as exchange relations from the point of view of monoidal categorification. Another motivation is the explicit tensor factorization of 
simple representations for $\mathfrak{g} = \mathfrak{sl}_2$ by Chari-Pressley (see \cite{CP2}). A crucial point is the interpretation 
of $q$-characters in terms of purely cluster theoretical data called $F$-polynomials (see \cite{HL1} and \cite{C} for other examples).

\begin{example} For example for $\mathfrak{g} = \mathfrak{sl}_2$, one has the $T$-system
$$[V_1(a)][V_1(aq^2)] = 1 + [W]$$
where $W$ is a simple evaluation module of dimension $3$. This has the form of an exchange relation.
\end{example}

\begin{example} Let $\mathfrak{g} = \mathfrak{sl}_3$. We define $\mathcal{M}$ as the 
monoidal Serre subcategory of  the category $\mathcal{C}$ generated by the four fundamental representations 
$V_1(1), V_1(q^2), V_2(q), V_2(q^3)$. Then we obtain a monoidal categorification $K(\mathcal{M})\simeq \mathcal{A}_Q$ with an initial seed
$$\begin{xymatrix}{[W_1]\ar[r]& [V_1(1)] \ar[r]
& [V_2(q)]  \ar@{->}@/^2pc/[ll] & [W_2]\ar[l]}\end{xymatrix}$$

\vspace{5mm}

We have written at each vertex the simple class corresponding to the initial cluster variable attached to this vertex. 
Here $W_1$ and $W_2$ are classes of $8$-dimensional evaluation representations in $\mathcal{M}$ and they correspond to frozen cluster variables (see Example \ref{evaljg}). The total number of cluster variables is $7$, as in the example \ref{ex7}.
As an application, we obtain that every simple representation in $\mathcal{M}$ can be factorized into tensor products of these $7$ representations. 
Note that there are infinitely many simple classes in this category $\mathcal{M}$.
\end{example}

\begin{example} Let $\mathfrak{g} = \mathfrak{sl}_2$. Let $K(\mathcal{C}_{\mathbb{Z}}^-)$ (resp. $K(\mathcal{C}_{\mathbb{Z}})$) be the subring of $K(\mathcal{C})$ generated by the $[V_1(q^{2r})]$ with $r\leq 0$ (resp. $r\in\mathbb{Z}$). This defines monoidal categories $\mathcal{C}_{\mathbb{Z}}^-$ and $\mathcal{C}_{\mathbb{Z}}$ of representations whose class in $K(\mathcal{C})$ belongs 
to the respective subrings $K(\mathcal{C}_{\mathbb{Z}}^-)$ and $K(\mathcal{C}_{\mathbb{Z}})$. Then $\mathcal{C}^-_\mathbb{Z}$ 
is a monoidal categorification with the initial seed:
$$[1 - z]\leftarrow [(1 - z)(1-zq^{-2})]\leftarrow [(1 - z)(1 - zq^{-2})(1 - zq^{-4})]\leftarrow \cdots$$
where $[P(z)]$ is the class of the simple module of Drinfeld polynomial $P(z)$. 

The monoidal category $\mathcal{C}_\mathbb{Z}$ is a monoidal categorification with the initial seed:
$$[1 - z]\leftarrow [(1 - z)(1-zq^{-2})]\rightarrow [(1 - zq^2)(1 - z)(1-zq^{-2})]\leftarrow \cdots$$
\end{example}

\begin{example}
For $\mathfrak{g} = \mathfrak{sl}_3$, let $K(\mathcal{C}_{\mathbb{Z}}^-)$ be the subring of $K(\mathcal{C})$ generated by the $[V_1(q^{2r})]$, $[V_2(q^{2r-1})]$ with $r\leq 0$. 
As above this defines a monoidal category $\mathcal{C}_{\mathbb{Z}}^-$ which is a monoidal categorification with the initial seed:
$$\xymatrix{  & \left[W_1^{(2)}\right] \ar[rd]  & &\left[ W_2^{(2)}\right]\ar[ll]\ar[rd] &  & \ar[ll]\cdots
\\ \left[W_1^{(1)}\right]\ar[ru]&&\left[ W_2^{(1)} \right]\ar[ll]\ar[ru]&& \left[ W_3^{(1)} \right]\ar[ru]\ar[ll]&\ar[l]\cdots
}$$
where $W_r^{(1)}$ (resp. $W_r^{(2)}$) is the simple module of Drinfeld polynomials $((1-z)(1-zq^{-2})\cdots (1-zq^{2r-2}),1)$ 
(resp. $(1,(1-zq^{-1})(1-zq^{-3})\cdots (1-zq^{2r-3}))$).
\end{example}

Statements for finite-dimensional representations in general type will be given below in a more general context (see Theorem \ref{mth}). 

There are many interesting examples of monoidal categorifications (the reader may refer to references in \cite{KICM, H}). Here is a new example.

\begin{example}\label{exso} Consider the category $\hat{\mathcal{O}}$ of modules in $\mathcal{O}$ whose action can be extended to the whole quantum affine algebra. Let $\mathfrak{g} = \mathfrak{sl}_2$ and $K(\hat{\mathcal{O}}_{\mathbb{Z}}^+)$ be the subring of $K(\mathcal{O})$ generated by the $[L^{\mathfrak{b}}(q^{-1}\frac{1 - q^{2r}z}{1 - q^{2r-2}z})]$ with $r > 0$ . This defines a monoidal category $\hat{\mathcal{O}}_{\mathbb{Z}}^+$. This is a monoidal categorification with the initial seed
$$\left[L^{\mathfrak{b}}\left(q^{-1}\frac{1-q^2z}{1-z}\right)\right]\leftarrow \left[L^{\mathfrak{b}}\left(q^{-2}\frac{1-q^4z}{1-z}\right)\right]\leftarrow \left[L^{\mathfrak{b}}\left(q^{-3}\frac{1-q^6z}{1-z}\right)\right]\leftarrow \cdots$$
These are evaluations of simple Verma modules (see Example \ref{evalj}). The first step mutations are for $r\geq 1$: 
$$\left[L^{\mathfrak{b}}\left(q^{-r}\frac{1 - q^{2r}z}{ 1 - z} \right)\right] \left[L^{\mathfrak{b}}\left(q\frac{1 - zq^{2r - 2}}{ 1 - zq^{2r}} \right)\right] 
= \left[L^{\mathfrak{b}}\left(q^{-r-1}\frac{1 - q^{2r+2}z}{ 1 - z} \right)\right]  + \left[L^{\mathfrak{b}}\left(q^{-r+1}\frac{1 - q^{2r - 2}z}{ 1 - z} \right)\right] .$$
In this lecture, we conjecture that such a statement holds for general types. 
\end{example}

There are other known interesting subcategories of $\mathcal{O}$ whose Grothendieck ring have a cluster algebra structure, 
see \cite{HL2, BC} for instance (in particular categories $\mathcal{O}^+_\mathbb{Z}$, $\mathcal{O}^-_{\mathbb{Z}}$ which satisfy a certain duality, see \cite{HL2, P}). 
The $QQ^*$-relations and the $QQ$-systems will also be seen as exchange relations in monoidal categorifications (Section \ref{sshcl}). Although it is not expected that the whole category 
$\mathcal{O}$ itself provides a monoidal categorification of a cluster algebra (because the tensor product is not compatible with 
the finite length property as discussed above), a closely related category will be monoidal categorification (a category defined in terms of shifted quantum affine algebras, see Section \ref{sshcl}). 

\section{Quantum Grothendieck rings}

Let $t$ be a new indeterminate.
The Grothendieck ring $K(\mathcal{C})$ has a natural non-commutative deformation, the quantum Grothendieck ring $(K_t(\mathcal{C}),*_t)$. The quantum Grothendieck ring is defined as a $\mathbb{Z}[t^{\pm 1/2}]$-subalgebra of a quantum torus, with a distinguished 
set of generators called fundamental quantum classes $[V_i(a)]_t$ ($i\in I$, $a\in\mathbb{C}^*$). There is a canonical 
surjection $\pi : K_t(\mathcal{C})\rightarrow K(\mathcal{C})$ such that $\pi(t) = 1$ and $\pi([V_i(a)]_t) = [V_i(a)]$ for any $i\in I$, $a\in\mathbb{C}^*$.

$K_t(\mathcal{C})$ was originally defined in simply-laced types \cite{N0, VV:qGro} from the point of view of geometric representation theory (categories of perverse sheaves on quiver varieties). It has also a purely algebraic construction that leads to its existence for all types \cite{H1}, as well as a categorical construction (the monoidal Jantzen filtrations \cite{FuH}) and a close connection with cluster algebras that will be discussed below.

\begin{example} In the $sl_2$-case, the subring $K_t(\mathcal{C}_{\mathbb{Z}})$ generated by the quantum classes $[V_1(q^{2r})]_{r\in\mathbb{Z}}$ is enough to describe the whole ring $K_t(\mathcal{C})$ and its ring structure is uniquely defined by
\begin{equation}\label{deux}[V(q^{r})]_t*[V(q^{r'})]_t = t^{\mathcal{N}(r'-r)}[V(q^{r'})]_t*[V(q^r)]_t +\delta_{r'-r,2}(1 - t^{-2})\text{ for $r' \geq r$,}\end{equation}
where  $\mathcal{N}(r) = \begin{cases}2(-1)^k\text{ if $r = 2 k > 0$} \\0\text{ otherwise.}\end{cases}$. 
\end{example}

Let $\mathfrak{n}$ be the "positive part" of the Lie algebra $\mathfrak{g}$, and $N$ be a unipotent Lie group with $\text{Lie}(N) = \mathfrak{n}$.

\begin{theorem}\cite{HLCr} For the simply-laced types, there is a natural monoidal subcategory $\mathcal{C}_Q$ so that the subalgebra $K_t(\mathcal{C}_Q)\subset K_t(\mathcal{C})$ generated by the fundamental classes $[V]_t$ with $V$ in $\mathcal{C}_Q$ satisfies  
$\mathbb{C}\otimes_\mathbb{Z} K_t(\mathcal{C}_Q)\simeq \mathcal{U}_t(\mathfrak{n})$. 
\end{theorem}

\begin{remark} The result in \cite{HLCr} (with Theorem \ref{klsl} below) also implies that the monoidal category $\mathcal{C}_Q$ is a categorification of $\mathbb{C}[N]$, together with its dual canonical basis in the sense of Lusztig \cite{L} (the basis of simple classes in $K(\mathcal{C}_Q)$ matches the dual canonical basis though an isomorphism $\mathbb{C}\otimes_{\mathbb{Z}} K(\mathcal{C}_Q)\simeq \mathbb{C}[N]$).
\end{remark}

\begin{example} Let $\mathfrak{g} = \mathfrak{sl}_3$. The algebra $K_t(\mathcal{C}_Q)$ is generated by $[V_1(1)]_t$, $[V_1(q^2)]_t$ and $[V_2(q)]_t$. We have 
$$[V_1(1)]_t * [V_1(q^2)]_t - t^{-1} [V_1(q^2)]_t * [V_1(1)]_t = (t^{1/2} - t^{-3/2}) [V_2(q)]_t$$
which implies that $E_1 = [V_1(1)]_t$ and $E_2 = [V_1(q^2)]_t$ generate the algebra and satisfy the quantum Serre relations 
$$[V_1(1)]_t * [V_1(1)]_t * [V_1(q^2)]_t - (t + t^{-1}) [V_1(1)]_t * [V_1(q^2)]_t * [V_1(1)]_t  + [V_1(q^2)]_t * [V_1(1)]_t * [V_1(1)]_t = 0,$$
$$[V_1(q^2)]_t * [V_1(q^2)]_t * [V_1(1)]_t - (t + t^{-1}) [V_1(q^2)]_t * [V_1(1)]_t * [V_1(q^2)]_t  + [V_1(1)]_t * [V_1(q^2)]_t * [V_1(q^2)]_t = 0.$$
\end{example}

Nakajima \cite{N0} proved that the quantum Grothendieck has also a remarkable canonical basis $[L]_t$ parametrized by simple modules $L$ in $\mathcal{C}$ and that can be computed by an analog of the Kazhdan-Lusztig algorithm which produces also analogs of Kazhdan-Lusztig polynomials. 

\begin{theorem}\cite{N0}\label{klsl} Suppose $\mathfrak{g}$ is of simply laced-type. Then we have $\pi([L]_t) = [L]$ for any simple $L$ in $\mathcal{C}$ and so the $q$-character of simple modules can be computed by an algorithm. Also the analog Kazhdan-Lusztig polynomials are positive.
\end{theorem}

The proof in \cite{N0} is a culmination of a series of papers on the geometric study of the representation theory of simply-laced quantum affine algebras in terms of quiver varieties, see \cite{Naams} and references there in.

Note that geometric character formulas for standard modules (tensor products of fundamental representations) have been obtained in \cite{HL3} for all types. 
Although the analog of Kazhdan-Lusztig polynomials are now known to be positive for general types \cite{FHOO}, the analog of the Kazhdan-Lusztig conjecture (that is $\pi([L]_t) = [L]$) is still open in general (in its form formulated for the non simply-laced types in the phD thesis of the speaker \cite{H0}). Recent advances, using cluster categorification, have been obtained for the non simply-laced types. These advances are partly based on relations to cluster algebras and to categorification results in \cite{KKOP} that will be also discussed below. Note that first results in the simply-laced types \cite{Bi, Qi} showed the compatibility between 
quantum Grothendieck rings and quantizations of cluster algebras. 
Then the following relies in part on the interpretation of the quantization of the Grothendieck ring for general types in terms of quantum cluster algebras. 

Recall that a simple module is said to be reachable if it corresponds to a cluster monomial in a monoidal categorification (here we use the 
monoidal categorification as defined in the next sections).

\begin{theorem}\cite{FHOOs, FHOO} For non-simply laced quantum affine algebras, the dimension (and character) of any simple reachable representation can be obtained from a Kazhdan-Lusztig algorithm.
\end{theorem}

\section{Shifted quantum groups}\label{shrep}

The shifted quantum affine algebras form a new class of quantum groups closely related to (quantized $K$-theoretical) Coulomb branches (in the sense of Braverman-Finkelberg-Nakajima \cite{bfn2}) that were introduced by Finkelberg-Tsymbaliuk \cite{FT}. The algebra $\mathcal{U}_q^\mu(\hat{\mathfrak{g}})$ can be seen as a variation of the quantum affine algebra $\mathcal{U}_q(\hat{\mathfrak{g}})$ depending on a shift parameter which is a coweight $\mu$ of the Lie algebra $\mathfrak{g}$ 
(for example in the $sl_2$-case, we can see $\mu$ as an integer).
For $\mu = 0$, $\mathcal{U}_q^0(\hat{\mathfrak{g}})$ is essentially $\mathcal{U}_q(\hat{\mathfrak{g}})$.
These algebras have a very interesting representation theory. In \cite{H} the speaker started a systematic study of the representations of shifted quantum affine algebras.
Let us illustrate this with some first results.

\begin{theorem}\cite{H} $\mathcal{U}_q^\mu(\hat{\mathfrak{g}})$ has a non-zero finite-dimensional representation if and only if $\mu$ is codominant.\end{theorem}

For a general $\mu$, $\mathcal{U}_q^\mu(\hat{\mathfrak{g}})$ contains a copy of the Cartan subalgebra $\mathcal{U}_q(\mathfrak{h}) = \mathbb{C}[k_i^{\pm 1}]_{i\in I}$ of $\mathcal{U}_q(\hat{\mathfrak{b}})$, and so we have a notion of weight space as well as 
 an abelian category $\mathcal{O}^\mu$ of representations of  $\mathcal{U}_q^\mu(\hat{\mathfrak{g}})$ defined as above (an object in $\mathcal{O}^\mu$ is not necessarily finite-dimensional but its weight spaces are finite-dimensional). 

\begin{theorem}\label{pah}\cite{H}  If $\mu$ is anti-codominant, then $\mathcal{U}_q^\mu(\hat{\mathfrak{g}})$ contains a subalgebra isomorphic to $\mathcal{U}_q(\hat{\mathfrak{b}})$ so that any simple module in $\mathcal{O}^\mu$ is simple as a $\mathcal{U}_q(\hat{\mathfrak{b}})$-module.
\end{theorem}

This implies a profound relation between the representation of shifted quantum affine algebras and quantum affine Borel algebras, and so a relation to quantum integrable models. This is illustrated in the example below.

\begin{example} The action of $\mathcal{U}_q(\hat{\mathfrak{b}})$ on the prefundamental representation $L_{i,a}^- = L^{\mathfrak{b}}(\Psib_{i,a}^{-1})$ can be extended to an 
action of $\mathcal{U}_q^{-\omega_i^\vee}(\hat{\mathfrak{b}})$, where $\omega_i^\vee$ is a fundamental coweight, to obtain $L(\Psib_{i,a}^{-1})$ (see Example \ref{expre}).
\end{example}

\begin{theorem}\cite{H} The simple objects $L(\Psib)$ in $\mathcal{O}^\mu$ are parametrized by $n$-tuples of rational fractions $\Psib = (\psi_i(z))_{i\in I}$ regular at $0$ satisfying $\text{deg}(\psi_i(z)) = \alpha_i(\mu)$ where $(\alpha_i)_{i\in I}$ are the simple roots of $\mathfrak{g}$.
\end{theorem}

\begin{example} As explained above, for $\mu$ anti-codominant, we recover the representation $L^{\mathfrak{b}}(\Psib)$ of $\mathcal{U}_q(\hat{\mathfrak{b}})$ (the underlying vector space, graded by weights, is the same). This includes the negative prefundamental representations and the finite-dimensional fundamental representations. 
\end{example}

\begin{example} For $i\in I$ and $a\in\mathbb{C}^*$, the simple representation $L(\Psib_{i,a})$ of $\mathcal{U}_q^{\omega_i^\vee}(\hat{\mathfrak{g}})$ is called a positive prefundamental representation (here $\omega_i^\vee$ is a fundamental coweight). It is of dimension $1$ ! (in contrast, recall that $L^{\mathfrak{b}}(\Psib_{i,a})$ is infinite-dimensional).
\end{example}

\begin{remark} The last example underlines the difference with the representation theory of ordinary quantum affine algebras: for any type of $\mathfrak{g}$, the algebra 
$\mathcal{U}_q^{\omega_i^\vee}(\hat{\mathfrak{g}})$ has a non-countable family of non-isomorphic $1$-dimensional representations. Another illustration: $\mathcal{U}_q^{-\omega_1^\vee}(\widehat{\mathfrak{sl}_2})$ admits evaluation morphisms not to the ordinary quantum group $\mathcal{U}_q(\mathfrak{sl}_2)$ 
as in Example \ref{evalj}, but to the $q$-oscillator algebra ($[E,F]$ equals $\frac{K}{q - q^{-1}}$ instead of $\frac{K - K^{-1}}{q - q^{-1}}$).
\end{remark}

Without loss of generality, we may assume that all roots and poles of the rational fractions as in Theorem \ref{pah} are integral powers of $q$. One obtains an abelian category
$$\mathcal{O}^{sh} = \bigoplus_{\mu\text{ coweight}} \mathcal{O}^\mu$$
which contains $\mathcal{O}^{sh,l}$ its subcategory of modules of finite length.

\begin{theorem}\cite{H, HZ} The sum of Grothendieck groups
$$K(\mathcal{O}^{sh}) = \bigoplus_{\mu} K(\mathcal{O}^\mu)$$
has a natural topological ring structure, with positive structure constants on the topological basis of simple classes. 

Moreover $K(\mathcal{O}^{sh,l})$ is a subring $K(\mathcal{O}^{sh})$. 
\end{theorem}

The ring structure is induced by a fusion product construction obtained from the Drinfeld coproduct as in \cite{HPr}. 
This ring can be considered as an analogue of a Grothendieck ring. In addition, it is commutative.
The fact that the subcategory of finite length representations is stable by fusion product is established in \cite{HZ}.

\section{Shifted quantum groups and cluster algebras}\label{sshcl}

Consider the subcategory $\mathcal{C}^{sh}\subset \mathcal{O}^{sh, l}$ of finite-dimensional representations of shifted quantum affine algebras. 
Its Grothendieck ring $K(\mathcal{C}^{sh})$ contains as a subring the Grothendieck ring $K(\mathcal{C})$ of finite-dimensional representations of the ordinary quantum affine algebra.

\begin{theorem}\label{mth}\cite{HL2, KKOP, H} $K(\mathcal{C}^{sh})$ is isomorphic to a cluster algebra $\mathcal{A}_{\Gamma_\infty}$, with an explicit quiver $\Gamma_\infty$. The initial cluster variables are classes of positive
prefundamental representations. The cluster monomials correspond to certain classes of simple modules.
\end{theorem}

\begin{example} For $\mathfrak{g} = \mathfrak{sl_2}$, the initial seed is 
$$\cdots\longrightarrow  \left[ L(\Psib_{1,q^{-2}})\right] \longrightarrow \left[L(\Psib_{1,1})\right]  \longrightarrow \left[L(\Psib_{1,q^{2}})\right]  \longrightarrow \cdots$$
(up to factors by invertible representations, that we omit in this lecture to avoid technicalities).
The first step mutations are given by the Baxter $TQ$-relations
$$[L(\Psib_{1,a})][V_1(a)] = [L(\Psib_{1,aq^2})] + [L(\Psib_{1,aq^{-2}})].$$
This connects cluster categorification to quantum integrable models (for general types, the first step mutations are $QQ^*$-relations 
as in Remark \ref{addqq}).
\end{example}

\begin{example}
For $\mathfrak{g} = \mathfrak{sl}_3$, the initial seed is
$$\xymatrix{ \cdots \ar[rr]& & [L(\Psib_{1,q^{-2}})] \ar[ld]  \ar[rr] & & [L(\Psib_{1,1})]\ar[rr]\ar[ld] && \cdots
\\ \cdots \ar[r]&[L(\Psib_{2,q ^{-3}})]\ar[rr]\ar[lu]&&[L(\Psib_{2,q^{-1}})]\ar[rr]\ar[lu]&&[L(\Psib_{2,q})]\ar[r]\ar[lu]&\cdots
}$$
\end{example}

\begin{remark} The results of \cite{K3, K1, K2} culminating in \cite{KKOP} established a monoidal categorification conjecture formulated in \cite{HL3} for finite-dimensional representations 
of ordinary quantum affine algebras.
\end{remark}

We use the Weyl group action and the generalized $QQ$-relations in \cite{FHw, FHgq} to obtain the following Theorem \ref{clsh}.
Indeed we introduce in \cite{FHw} the operators $\Theta_i$ ($i\in I$) on a sum 
$$\Pi = \bigoplus_{w\in W}\tilde{\mathcal{Y}}^w$$ 
of completions $\tilde{\mathcal{Y}}^w$ of $\mathcal{Y} = \mathbb{Z}[Y_{i,a}^{\pm 1}]_{i\in I, a\in\mathbb{C}^*}$, parametrized by the Weyl group $W$ of $\mathfrak{g}$. 
The operators are characterized by the formula ($i,j\in I$ and $a\in\mathbb{C}^*$): 
$$\Theta_i(Y_{j,a}) = Y_{j,a}A_{i,aq^{-1}}^{-\delta_{i,j}}\frac{\Sigma_{i,aq^{-3}}^{\delta_{i,j}}}{\Sigma_{i,aq^{-1}}^{\delta_{i,j}}},$$
\begin{equation}\label{wa}\text{ where }A_{i,a} = Y_{i,aq_i^{-1}}Y_{i,aq_i}\prod_{j|C_{j,i} = -1} Y_{j,a}^{-1}\prod_{j|C_{j,i} = -2}Y_{j,aq^{-1}}^{-1}Y_{j,aq}^{-1}\prod_{j|C_{j,i} = -3}Y_{j,aq^{-2}}^{-1}Y_{j,a}^{-1}Y_{j,aq^2}^{-1}\end{equation}
are the root monomials of \cite{Fre} and $\Sigma_{i,a}$ is the solution in $\Pi$ of the $q$-difference equation
$$\Sigma_{i,a} = 1 + A_{i,a}^{-1}\Sigma_{i,aq^{-2}}.$$
Now $\mathcal{Y}$ embeds in $\Pi$ diagonally, and we have the following analogue of a classical result.

\begin{theorem}\cite{FHw} The $\Theta_i$ are involutions and define a Weyl group action on $\Pi$ (with the simple reflexions acting as $\Theta_i$). Moreover, the ring of invariants
$\mathcal{Y}^W$ equals the ring of $q$-characters in $\mathcal{Y}$.
\end{theorem}

Now starting from $Q_a^{\omega_i^\vee} = [L(\Psib_{i,a})]$, and using this Weyl group action, we introduce in \cite{FHgq} for general $w\in W$, $i\in I$ new elements 
$Q_a^{w(\omega_i^\vee)}$ in the Grothendieck ring $K(\mathcal{O}^{sh})$ which satisfy a generalized version of the $QQ$-system for each $w\in W$ (with the $Q_{a}^{w(\omega_i^\vee)}$ for the $Q$-variables and the $Q_{a}^{(ws_i)(\omega_i^\vee)}$ for the $\tilde{Q}$-variables).

These solutions of the generalized $QQ$-systems lead to the construction of remarkable cluster variables (the $Q$-variables) used to construct initial seeds for the following result.

\begin{theorem}\label{clsh}\cite{GHL} For the simply-laced types, the Grothendieck ring $K(\mathcal{O}^{sh})$ is isomorphic to (a completion of) a cluster algebra $\mathcal{A}_{\Gamma_\infty'}$, with an explicit quiver $\Gamma_\infty'$.
\end{theorem}

\begin{example} For $\mathfrak{g} = \mathfrak{sl}_2$, the new quiver $\Gamma_\infty'$ is obtained from the older quiver $\Gamma_\infty$ by just inverting the direction of one arrow: 
$$\cdots\longrightarrow  \left[ L(\Psib_{1,q^2})\right] \longrightarrow \left[ L(\Psib_{1,1})\right] \longleftarrow \left[ L(\Psib_{1,q^{-2}}^{-1})\right] \longrightarrow \left[ L(\Psib_{1,q^{-4}}^{- 1})\right]\longrightarrow \cdots$$
The mutation at the vertex corresponding to $L(\Psib_{1,1})$ is 
$$\cdots\longrightarrow  \left[ L(\Psib_{1,q^2})\right] \longleftarrow \left[ L(\Psib_{1,1}^{-1})\right] \longrightarrow \left[ L(\Psib_{1,q^{-2}}^{-1})\right] \longrightarrow 
\left[ L(\Psib_{1,q^{-4}}^{-1})\right]\longrightarrow \cdots$$
as we have the remarkable $QQ$-relation 
\begin{equation}\label{qqexsl}[L(\Psib_{1,1})][L(\Psib_{1,1}^{-1})] = 1 + [-\omega_1]^2 [L(\Psib_{1,q^2})][L(\Psib_{1,q^{-2}})].\end{equation}
Here $Q_a^{\omega_1^\vee} = [L(\Psib_{1,a})]$ and $Q_a^{-\omega_1^\vee} = [L(\Psib_{1,aq^{-2}}^{-1})]$.
\end{example}

\begin{example}
For $\mathfrak{g} = \mathfrak{sl}_3$, the new quiver $\Gamma_{\infty}'$ is also built from the periodic quiver $\Gamma_{\infty}$, but it contains now $3$ quadrilaterals: 
$$
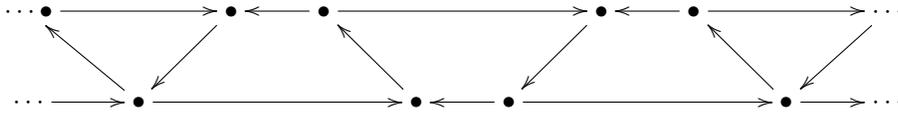
\xymatrix{ 
\cdots \bullet\ar[rr]&                      & \bullet \ar[ld] & \bullet \ar[l]\ar[rrr] &                &                     & \bullet\ar[ld] &\bullet \ar[rr]\ar[l]& &\cdots \ar[ld]
\\           
\cdots \ar[r]        &\bullet\ar[rrr]\ar[lu]&                 &                        &\bullet\ar[lu]   & \bullet \ar[l]\ar[rrr]  &&& \bullet\ar[lu] \ar[r]&\cdots
}$$
The initial cluster variables are the $[L(\Psib_{1,q^{2r}})]$,  $[L(\Psib_{2,q^{2r - 1}})]$, $[L(\Psib_{1,q^{2s}}^{-1})]$,  $[L(\Psib_{2,q^{2s - 1}}^{-1})]$, with $r\geq 0$, $s\leq -2$, as well as 
$L(\tilde{\Psib}_{1,q^{-2}})$ and $L(\tilde{\Psib}_{1,q^{-4}})$ (recall (\ref{tps})). These are $Q$-variables as here $Q_a^{\omega_1^\vee} = [L(\Psib_{1,a})]$, $Q_a^{\omega_2^\vee - \omega_1^\vee} = [L(\Psib_{2,aq^{-1}}\Psib_{1,aq^{-2}}^{-1})]$, $Q_a^{-\omega_2^\vee} = [L(\Psib_{2,aq^{-3}}^{-1})]$, $Q_a^{\omega_2^\vee} = [L(\Psib_{2,a})]$, $Q_a^{\omega_1^\vee - \omega_2^\vee} = [L(\Psib_{1,aq^{-1}}\Psib_{2,aq^{-2}}^{-1})]$, $Q_a^{-\omega_1^\vee} = [L(\Psib_{1,aq^{-3}}^{-1})]$.\end{example}

The quantization of this structure is discussed in \cite{Pa}. 

\begin{remark}\label{remqqsh} By \cite{H0}, Equation (\ref{qqexsl}) generalizes to all finite types and the $QQ$-relation in Theorem \ref{qqo} also holds in $K_0(\mathcal{O}^{sh})$, with the simple classes $[L(\Psib_{i,a})] = Q_a^{\omega_i^\vee}$ and $[L(\tilde{\Psib}_{i,aq^{-2}})] = Q_a^{s_i(\omega_i^\vee)}$ which are $Q$-variables. This is one of the 
motivations for the next Conjecture.
\end{remark}

In general, we conjecture the following, which can be seen as a generalization of Theorem \ref{mth} to the category $\mathcal{O}^{sh}$.

\begin{conjecture}\label{clc}\cite{GHL} All cluster monomials in $\mathcal{A}_{\Gamma_\infty'}$ correspond to classes of simple objects in $\mathcal{O}^{sh}$ through the  isomorphism 
in Theorem \ref{clsh}.
\end{conjecture}

This general Conjecture is, for the moment, only known for $\mathfrak{g} = sl_2$ in \cite{GHL}.

\section{Truncated shifted quantum groups}

Let us now discuss truncated shifted quantum affine algebras, quotients $\mathcal{U}_{q,{\bf Z}}^\mu(\hat{\mathfrak{g}})$ of $\mathcal{U}_q^\mu(\hat{\mathfrak{g}})$ 
depending on a truncation (or flavor) parameter ${\bf Z} = (Z_i(z))_{i\in I}$ which is a collection of polynomials with constant term $1$. These truncations defined 
in \cite{FT} are closely related to the geometric context of Coulomb branches (in their $K$-theoretical version).
 
The truncated shifted quantum affine algebras appear also naturally from the point of view of quantum integrable systems (Section \ref{qims}). Indeed, the defining relations of the 
truncated shifted quantum affine algebras are expressed in terms of the polynomiality of certain $A$-operators $A_i^\pm(z)$ ($i\in I$) which have coefficients in the 
shifted quantum affine algebra. The $A$-operators can be also obtained as limits of transfer-matrices, not associated to infinite-dimensional representations as the $Q$-operators, 
but associated to finite-dimensional representations.

We will use methods related to the Baxter polynomiality in quantum integrable systems to approach the representation theory of truncated shifted quantum affine algebras of general types. 

\begin{example}Let $\mathfrak{g} = sl_2$. We have the $Q$-operator $\mathcal{Q}_u(z) = \mathcal{T}_{R_1^+}(z,u)$ 
obtained in \cite{FH} from the (infinite-dimensional) $\mathcal{U}_q(\hat{\mathfrak{b}})$-prefundamental representation $R_{i,1}^+$ of \cite{HJ} (it is a dual of $L_{i,1}^+$). 
We have the limit $u\rightarrow 0$ of the $Q$-operator computed in \cite[Section 5.3]{FH} in terms of Drinfeld generators $h_{-m}$ : 
$$(\mathcal{Q}_u(z))_{u\rightarrow 0} = T^-(z) = \text{exp}\left(\sum_{m > 0} \frac{h_{-m} z^m}{[m]_q (q^m + q^{-m})}\right).$$
In fact, the operator $A^-(z)$ of \cite{FT} is equal, up to a scalar function in $z$, to
$$A^-(z) = \tilde{k} \frac{T^-(z^{-1}q^2)}{T^-(z^{-1})},$$
with $\tilde{k}^2 = k$. It can also be obtained as a limit of a transfer-matrix in the following way. We have the $T$-operator
$$\mathcal{T}_u(z) = t_{V(1)}(z,u).$$
The Baxter $TQ$-relation (\ref{tqr}), in terms of transfer-matrices, reads
$$\mathcal{T}_u(z)\mathcal{Q}_u(zq) = u\tilde{k}^{-1}\mathcal{Q}_u(z q^{-1}) + u^{-1}\tilde{k} \mathcal{Q}_u(zq^3).$$
Then for $u\rightarrow 0$ we obtain
$$(u\mathcal{T}_u(z))_{u\rightarrow 0} = \tilde{k}\frac{T^-(zq^3)}{T^-(zq)}$$
which equals $A^-(z^{-1}q^{-1})$ up to a scalar function. 
The proof of the polynomiality of $Q$-operator in \cite{FH} can be extended to $T$-operators as established in \cite{FJMM}. 
Hence the polynomiality of the $A$-operators can be studied from these methods.
\end{example}

Recall that it is established in \cite{FH} that the eigenvalues of $Q$-operators on finite-dimensional representations 
can be described in terms of certain polynomials, generalizing Baxter's polynomials 
associated to the $XXZ$-model.  Moreover, this Baxter polynomiality implies the polynomiality of certain series of Cartan-Drinfeld elements acting on 
finite-dimensional representations \cite{FH}. For instance, it implies the following two Theorems.

\begin{theorem}\cite{H} Each truncation has a finite number of simple representation in the category $\mathcal{O}^{sh}$.
\end{theorem}

For a fixed ${\bf Z}$, we denote by $\mathcal{O}^{sh}_{\bf Z} =  \sum_\mu \mathcal{O}^\mu_{\bf Z}$ the sum for various coweights $\mu$ of the categories $\mathcal{O}^\mu_{\bf Z}$ of representations in $\mathcal{O}^\mu$ that descend to $\mathcal{U}_{q,{\bf Z}}^\mu(\hat{\mathfrak{g}})$.

\begin{remark}\label{const} From the defining relation of the truncated shifted quantum affine algebras, if $L(\Psib)$ is in 
$\mathcal{O}_{{\bf Z}}^\mu$, then for any $i\in I$ we have $Z_i(0)Z_i(\infty) = \Psi_i(0)\Psi_i(\infty)$ up to a sign, where 
$Z_i(\infty) = \text{lim}_{z\rightarrow \infty} z^{-\text{deg}(Z_i)}Z_i(z)$ and $\Psi_i(\infty) = \text{lim}_{z\rightarrow \infty} z^{-\text{deg}(\Psi_i)}\Psi_i(z)$. 
These conditions determine the constant part of $(\Psi_i(0))_{i\in I}$ (up to a $4$th root if $1$). Hence for clarity we will omit the constant parts in the following.
\end{remark}

\begin{theorem}\label{evth}\cite{H0, Z, HZ} Each simple representation $L(\Psib)$ in $\mathcal{O}^{sh}$ descends to a truncation (moreover, an explicit truncation  parameter ${\bf Z}$ 
so that $L(\Psib)$ belongs to $\mathcal{O}_{{\bf Z}}^{sh}$ can be computed from $\Psib$). 
\end{theorem}

\begin{example} (i) For a truncation parameter ${\bf Z} = (Z_i(z))_{i\in I}$, let $\lambda = \sum_{i\in I}\text{deg}(Z_i)\omega_i^\vee$. 
Then $\mathcal{O}^\lambda_{\bf Z}$ contains a unique simple object which is $L({\bf Z})$. It is a one-dimensional module.

(ii) Let ${\bf Z} = \Psib_{i,a}$. Then $\mathcal{O}^{\omega_i^\vee - \alpha_i^\vee}_{\bf Z}$ contains a unique simple object which is 
precisely $L(\tilde{\Psib}_{i,aq_i^{-2}})$ the simple representation that occurs in the $QQ$-system in Theorem \ref{qqo}.
\end{example}

A partial ordering is defined in \cite{H} : $\Psib\preceq \Psib'$ if and only if $\Psib'\Psib^{-1}$ is a monomial in the 
$\Lambda_{i,a} = \Psib_{i,aq_i}(\tilde{\Psib}_{i,aq_i^{-1}})^{-1}$, $i\in I, a\in\mathbb{C}^*$.

\begin{theorem}\label{H0} If $L(\Psib)$ is in $\mathcal{O}_{\bf Z}^{sh}$, then $\Psib \preceq {\bf Z}$.
\end{theorem}

\noindent Using the same tools as for the proof of these results, one obtains the following.

\begin{prop}\label{extr} Let ${\bf Z}$, ${\bf Z'}$ truncation parameters (that is collections of polynomials as above).

(i) $L({\bf Z'})$ is in $\mathcal{O}_{\bf Z}^{sh}$ if and only if ${\bf Z'}\preceq {\bf Z}$.

(ii) In this case, we have $\mathcal{O}_{\bf Z'}^{sh}\subset \mathcal{O}_{\bf Z}^{sh}$.
\end{prop}

\begin{proof}
(i) follows from the fact that, as the components of ${\bf Z'}$ are polynomials, $L({\bf Z'})$ is one-dimensional. So this is just a statement on the eigenvalues of the $A$-operators. 

(ii) If ${\bf Z'} = {\bf Z} \Lambda$ with $\Lambda$ a product of various of $\Lambda_{i,a}$, the truncation relations for ${\bf Z'}$ imply the truncation relations for ${\bf Z}$.
\end{proof}

\section{Representations of truncations} This last Section is more technical as we discuss Conjectures. For the non-simply-laced types, we proposed in \cite{H0} a parametrization of simple representations in $\mathcal{O}_{\bf Z}^{sh}$ in terms of specializations of commutative analogs of deformed $W$-algebras.

\begin{remark} (1) The shifted quantum affine algebras have rational analogs called shifted Yangians.
For the simply-laced types, simple representations of truncated shifted Yangians have been parametrized in terms of Nakajima monomial crystals \cite{ktwwy2} (this parametrization can be understood in the context of symplectic duality, more precisely from the equivariant version of the Hikita conjecture for the symplectic duality formed by an affine Grassmannian slice and a quiver variety). As for ordinary (non shifted) quantum groups \cite{GTL}, the representations of shifted Yangians and shifted quantum affine algebras are expected to be closely related.
This is true at the level of truncations by \cite{VV} and the parametrization also holds for the simply-laced truncated shifted quantum 
affine algebras. Note that a bijection is given in \cite{nw} between certain simple representations of a non simply-laced quantized Coulomb branch and those for the simply-laced types. 

(2) The prefundamental representations have also been recently studied by using other approaches, for example in \cite{Ne, VVO}. 
\end{remark}

In our approach to general types, we use a limit obtained from interpolating $(q,t)$-characters. 
The latter were defined by Frenkel and the speaker as a commutative 
incarnation of Frenkel-Reshetikhin deformed $\mathcal{W}$-algebras \cite{Frw} which appeared in the context of the geometric Langlands program 
\cite{FH0, FHR}. They lead to an interpolation between the Grothendieck ring 
$\text{Rep}(\mathcal{U}_q(\hat{\mathfrak{g}}))$ of finite-dimensional representations of
 $\mathcal{U}_q(\hat{\mathfrak{g}})$ (at $t = 1$) and the Grothendieck ring $\text{Rep}(\mathcal{U}_t({}^L\hat{\mathfrak{g}}))$ 
of finite-dimensional representations of the Langlands dual
 quantum affine algebra $\mathcal{U}_t({}^L\hat{\mathfrak{g}})$ (at $q = \epsilon$ a certain root of $1$): 
$$\xymatrix{ &\mathfrak{K}_{q,t}\ar@{-->}[dr]^{q = \epsilon}\ar@{-->}[dl]^{t = 1}&   
\\ \text{Rep}(\mathcal{U}_q(\hat{\mathfrak{g}})) & &  \text{Rep}(\mathcal{U}_t({}^L\hat{\mathfrak{g}}))
}$$
Here $\mathfrak{K}_{q,t}$ is the ring of interpolating $(q,t)$-characters. It is generated by fundamental interpolating $(q,t)$-characters
$X_{i,q,t}(a)$, $i\in I, a\in\mathbb{C}^*$ that can be computed by a deformation of an 
algorithm of Frenkel-Mukhin \cite{FM}.

To describe our parametrization of simple representations in $\mathcal{O}_{{\bf Z}}^{sh}$, we found in \cite{H0} that it is relevant to use a mixed specialization of interpolating $(q,t)$-characters (with $t = 1$ for variables but $q = \epsilon$ for coefficients). By our specialization process, for each $i\in I, a\in \mathbb{C}^*$, we obtain a fundamental 
$$\chi_{i,a} \in \mathbb{Z}[Z_{j,aq^r}^{\pm 1}]_{j\in I, r\in\mathbb{Z}}.$$
Informally, this $\chi_{i,a}$ can be seen as obtained from the $q$-character of a fundamental representation of the Langlands dual twisted quantum affine algebra $\mathcal{U}_t({}^L\hat{\mathfrak{g}})$, 
where each term has been replaced through interpolation by a monomial which occurs in the $q$-character of a representation of $\mathcal{U}_q(\hat{\mathfrak{g}})$. $\chi_{i,a}$ can be computed algorithmically.

\begin{example}\label{sexb} (i) For $\mathfrak{g}$ of simply-laced type, $\chi_{i,a}$ is just the $q$-character of the finite-dimensional 
fundamental representation $V_i(a)$ of $\mathcal{U}_q(\hat{\mathfrak{g}})$, where we have set $Z_{i,a} = Y_{i,a}$.

(ii) Let $\mathfrak{g}$ in type $B_2$ with $r_1 = 2$ and $r_2 = 1$. Then we have
$$\chi_{1,a} = Z_{1,a} + Z_{1,aq^4}^{-1}Z_{2,aq^2} + Z_{2,aq^4}^{-1}Z_{1,aq^2} + Z_{1,aq^6}^{-1},$$
$$\chi_{2,a} = Z_{2,a} + Z_{2,aq^2}^{-1}Z_{1,a}Z_{1,aq^2} + Z_{1,a}Z_{1,aq^6}^{-1} Z_{2,aq^2}^{-1}Z_{2,aq^4} 
+ Z_{1,aq^2}Z_{1,aq^4}^{-1} +  Z_{1,aq^6}^{-1}Z_{1,aq^4}^{-1}Z_{2,aq^4} + Z_{2,aq^6}^{-1}.$$
The respective number of terms, $4$ and $6$, are the respective dimensions of the fundamental representations of the twisted quantum 
affine algebra $\mathcal{U}_q(A_3^{(2)})$ (Langlands dual to $\mathcal{U}_q(B_2^{(1)})$).
\end{example}

\begin{remark} A priori $\chi_{i,a}$ may have infinitely many terms, but we conjecture that the number of terms is finite, that is the 
algorithm to compute $\chi_{i,a}$ stops.
\end{remark}

Now consider a truncation parameter ${\bf Z} = (Z_i(z))_{i\in I}$. 
Then we define
$$\chi({\bf Z}) = \prod_{i\in I, a\in\mathbb{C}^*}\chi_{i,a^{-1}}^{u_{i,a}},$$ 
where $u_{i,a}$ is the multiplicity of $a^{-1}$ as a root of $Z_i$. For $M$ a monomial in $\chi({\bf Z})$, we set $\Psib_M = \prod_{i\in I, a\in\mathbb{C}^*}\Psib_{i,a^{-1}}^{z_{i,a}}$ where $z_{i,a}$ is the power of $Z_{i,a}$ in $M$ ($\Psib_M$ is defined up to an explicit constant, as discussed in Remark \ref{const}). In this situation, we say that $\Psib_M$ comes from $\chi({\bf Z})$.

Now we reformulate and adjust the conjecture in \cite{H0}.

\begin{conjecture}\label{ctru} (i) The representation $L(\Psib)$ belongs to $\mathcal{O}_{\bf Z}^{sh}$ if $\Psib$ comes from $\chi({\bf Z})$ .

(ii) The representation $L(\Psib)$ belongs to $\mathcal{O}_{\bf Z}^{sh}$ only if there is a finite sequence $({\bf Z}_0 = {\bf Z} , {\bf Z}_1,\cdots , {\bf Z}_N)$ of truncation parameters such that ${\bf Z}_{k+1}$ comes from $\chi({\bf Z}_k)$ for all $0\leq k < N$ and $\Psib$ comes from $\chi({\bf Z}_N)$.
\end{conjecture}

\begin{remark} (1) For $\Psib = {\bf Z'}$ a truncation parameter, (i) is true by Proposition \ref{extr} as the condition implies that ${\bf Z'}\preceq {\bf Z}$.

(2) The statement in (i) is equivalent to the converse of (ii). But the converse of (i) is untrue in general (see the counter example \ref{cex}). That is why (ii) is reformulated in comparison to \cite{H0}. Note that, for the simply-laced types, (i) and its converse have been established (at least for truncated shifted Yangians) by \cite{ktwwy2} 
in the context of categorification of action of Lie algebras. 
This implies that for simply-lace types, the converse of (ii) is (i). For these types, $\chi({\bf Z})$ is the $q$-character of a standard module, and the set of its monomials is a monomial crystal which 
is a product of fundamental monomial crystals in the sense of Nakajima \cite{Nc} (see \cite{HNc} for a generalization). 
The formulation can be simplified: $L(\Psib_M)$ belongs to $\mathcal{O}_{\bf Z}^{sh}$ if and only if $\Psib_M$ comes from $\chi({\bf Z})$. In particular, ${\bf Z}' = {\Psib_{M'}} \preceq {\bf Z} = {\Psib_M}$ if and only if $M'$ dominant occurs in $\chi_M$ $q$-characters of a standard module ($M'$ dominant means that the powers of the variables $Y_{j,b}$ are positive). 
This implies the following: consider a dominant $M'\leq M$ (for the Nakajima partial ordering, that is $M'M^{-1}$ is a product of various $A_{j,b}$ as in (\ref{wa})). Then $M'$ occurs in the $q$-character of the standard module associated to $M$. 

(3) For general types, from the discussions above, the condition (ii) can be replaced by:  there is ${\bf Z'}\preceq {\bf Z}$ truncation parameter such that $\Psib$ comes from a monomial in $\chi({\bf Z'})$. 
This formulation is less explicit as we have to determine the truncation parameters that satisfy ${\bf Z'}\preceq {\bf Z}$.

(4) An alternative formulation of Conjecture \ref{ctru} is given in \cite{VV} in terms of crystals.

(5) Theorem \ref{evth} is a main evidence as it gives a truncation parameter as predicted by Conjecture \ref{ctru}.
\end{remark}

\begin{example} We continue Example \ref{sexb} in type $B_2$.

(1) First let us set $\lambda = \omega_2^\vee$, $Z_1(z) = 1$, $Z_2(z) = 1 - z$. We list the parameters of the $6$ simple modules in $\mathcal{O}_{\bf Z}^{sh}$:

$\mu = \omega_2^\vee$: $\Psib_{2,1}$ ; $\mu = 2\omega_1^\vee - \omega_2^\vee$: $\Psib_{1,q^{-2}}\Psib_{1,1}\Psib_{2,q^{-2}}^{-1}$ ; $\mu = \omega_2^\vee - 2\omega_1^\vee$: $\Psib_{1,q^{-6}}^{-1}\Psib_{1,q^{-4}}^{-1}\Psib_{2,q^{-4}}$ ; $\mu = -\omega_2^\vee$: $\Psib_{2,q^{-6}}^{-1}$;

$\mu = \omega_2^\vee - 2\omega_1^\vee$: $\Psib_{1,1}\Psib_{1,q^{-6}}^{-1}\Psib_{2,q^{-4}}\Psib_{2,q^{-2}}^{-1}$
 and $\Psib_{1,q^{-2}}\Psib_{1,q^{-4}}^{-1}$.

(2) Set $\lambda = \omega_1^\vee$, $Z_1(z) = 1 - z$, $Z_2(z) = 1 $. We have $4$ simple modules in $\mathcal{O}^{\bf Z}$:

$\mu = \omega_1^\vee$: $\Psib_{1,1}$ ; $\mu = \omega_2^\vee - \omega_1^\vee$: $\Psib_{1,q^{-4}}^{-1}\Psib_{2,q^{-2}}$ ; $\mu  = \omega_1^\vee - \omega_2^\vee$: $\Psib_{1,q^{-2}}\Psib_{2,q^{-4}}^{-1}$ ; $\mu  = -\omega_1^\vee$: $\Psib_{1,q^{-6}}^{-1}$.

\end{example}

\begin{example}\label{cex} Let us set ${\bf Z} = ((1- z q^{-6}), (1-z)(1-zq^{-2})(1-zq^{-6}))$. Then
$${\bf Z'} = (1-z, (1-zq^{-4})(1-zq^{-6})) = {\bf Z} \Lambda_{2,q^{-1}}^{-1}\Lambda_{1,q^{-4}}^{-1}$$
is associated to the monomial $Z_{1,q^6} (Z_{1,1}Z_{1,q^6}^{-1} Z_{2,q^2}^{-1}Z_{2,q^4})    Z_ {2,q^2}Z_{2,q^6} = Z_{1,1}Z_{2,q^4}Z_{2,q^6}$ in $\chi({\bf Z})$. Besides 
$${\bf Z''} = ((1-z)(1-zq^{-4})(1-zq^{-6}),1) = {\bf Z'}\Lambda_{2,q^{-5}}^{-1}$$
is associated to the monomial $Z_{1,1} Z_{2,q^4} (Z_{2,q^4}^{-1}Z_{1,q^4}Z_{1,q^6}) = Z_{1,1} Z_{1,q^4}Z_{1,q^6} $ in $\chi({\bf Z'})$.

Hence, $L({\bf Z''})$ belongs to $\mathcal{O}_{{\bf Z}}^{sh}$ as predicted by the conjecture. But ${\bf Z''}$ does not come from a monomial in $\chi({\bf Z})$. 
\end{example}

We have the following compatibility property with the ring structure. 

\begin{proposition}\cite{H0} For truncation parameters ${\bf Z}$, ${\bf Z'}$, we have 
$$K(\mathcal{O}_{{\bf Z}}^{sh})K(\mathcal{O}_{{\bf Z'}}^{sh})\subset K(\mathcal{O}_{{\bf Z Z'}}^{sh}),$$
where ${\bf Z Z'}$ is obtained by multiplying the polynomials (coordinate by coordinate).
\end{proposition}

An open and important question is to understand the compatibility with the cluster algebra structures 
(and with the Conjecture \ref{clc}). For example, by the last Proposition, for the truncation parameter
$${\bf Z} = \left(\prod_{j,C_{i,j} = -1}\Psib_{j,aq_i} \right)\left(\prod_{j,C_{i,j} = -2}\Psib_{j,a}\Psib_{j,aq^2} \right)\left(\prod_{j,C_{i,j} = -3}\Psib_{j,aq^{-1}}\Psib_{j,aq}\Psib_{j,aq^3} \right).$$
the $QQ$-relation in Remark \ref{remqqsh} is a relation in $K(\mathcal{O}^{sh}_{\bf Z})$.

\section*{Acknowledgments.} I would like to thank heartily all my mentors, collaborators and students for inspiring discussions over the years. Without them this work would not have been possible. I am also very grateful to Edward Frenkel, Bernhard Keller, Bernard Leclerc, Hiraku Nakajima and Eric Vasserot for comments on a preliminary version of this lecture.

\end{document}